\documentclass[12pt,a4paper]{article}

\usepackage{amssymb}
\usepackage{amsmath}
\usepackage{indentfirst}
\usepackage{t1enc}
\usepackage{array}
\usepackage{graphpap}
\usepackage[mathscr]{eucal}
\usepackage{enumerate}
\usepackage{tabularx}
\usepackage{graphicx}
\usepackage[colorlinks=true,linkcolor=blue,citecolor=blue,
  urlcolor=blue]{hyperref}
\usepackage[hyphenbreaks]{breakurl}
\usepackage{amsthm}
\usepackage{scalerel} %Matematikai képletek nagyításához \scaleto parancs
\usepackage{comment}

\usepackage{tikz} 
\usepackage{pgfplots} 
\usetikzlibrary{shapes,arrows}

\usepackage[latin2,utf8]{inputenc}

\begin{document}

\theoremstyle{plain}

%\newtheorem{thm}{\textsc{T\'etel}}[section]
%\newtheorem{thm}{\textsc{T\'etel}}
%\newtheorem{lem}{\textsc{Lemma}}
%\newtheorem{cor}{\textsc{K\"ovetkezm\'eny}}
%\newtheorem{cors}{\textsc{K\"ovetkezm\'enyek}}
%\newtheorem{rmk}{\textsc{Megjegyzés}}
%\newtheorem{dfn}{\textsc{Definíció}}
%\newtheorem{cns}{\textsc{Konstrukció}}
%\newtheorem{all}{\textsc{Állítás}}
%\newtheorem{fel}{\textsc{Feladat}}
%\newtheorem{sejt}{\textsc{Sejtés}}
%\newtheorem{megj}{\textsc{Megjegyzés}}
%\newtheorem{prop}{\textsc{Propozíci\'o}}
%\newtheorem{conj}{\textsc{Sejt\'es}}
%\newtheorem{remark}{\textsc{Megjegyz\'es}}
%\newtheorem{segedt}{\textsc{Segédtétel}}
%\newtheorem{prob}{\textsc{Probléma}}

%Angolban majd:
\theoremstyle{plain}
\newtheorem{thm}{Theorem}
\newtheorem{cor}{Corollary}
\newtheorem{lem}{Lemma}
\newtheorem{dfn}{Definition}
\newtheorem{rmk}{Remark}
\newtheorem{cns}{Construction}
\newtheorem{exm}{Example}
\newtheorem{prs}{Proposition}
\newtheorem{prob}{Problem} 
\newtheorem{ntn}{Notation}
\newtheorem{conj}{Conjecture}
\newtheorem{prop}{Proposition}
\newtheorem{fel}{Excercise}

\newcommand{\gc}{\operatorname{lnko}}
\newcommand{\mc}{\mathcal}
\newcommand{\mr}{\mathscr}
\newcommand{\ab}[1]{\left\vert{#1}\right\vert}
\newcommand{\zj}[1]{\left({#1}\right)}
\newcommand{\norma}{N_{\mathbb F_{q^n}/\mathbb F_q}}
\newcommand{\normawp}{N_{\mathbb F_{p^n}/\mathbb F_p}}
\newcommand{\normfield}{\mc N_{\mathbb F_{q^n}/\mathbb F_q} (f)}

\renewcommand{\baselinestretch}{1.38}
\renewcommand\thefootnote{\relax}

\newcommand{\bs}{\bigskip}
\newcommand{\bi}{\bigskip\noindent}
\newcommand{\red}[1]{\textcolor{red}{#1}}
\newcommand{\lb}[1]{\label{#1}}
\allowdisplaybreaks[1] 

\normalsize

\sloppy

%\numberwithin{equation}{section}
%\setcounter{tocdepth}{1}

%Page 1
\title{Bounded gaps and perfect power gaps in sequences of consecutive primes}

\author{Katalin Gyarmati}
\date{}

\footnotetext{\noindent 2020 Mathematics Subject 
Classification: Primary: 11N05.\\
\indent Keywords and phrases: prime gaps, squares and powers.\\
\indent Research supported by the Hungarian National Research Development and 
Innovation Funds KKP133819.}
%%\thanks{}
\maketitle

\begin{abstract}
  We study whether several consecutive prime gaps can all be
  relatively large at the same time, or is it possible that all are squares
  or perfect powers, or perhaps none of them are squares?
  A few related results and problems are also presented.
\end{abstract}

\section{Introduction}

It has long been conjectured that there are infinitely many twin primes. Building on classical sieving techniques, significant recent progress has been made toward proving this conjecture by Goldston, Pintz, and Yıldırım \cite{Goldston1, Goldston2, Goldston3}, Zhang \cite{Zhang}, and Maynard \cite{Maynard1}.
The most notable result was proved
within the framework of the Polymath
Project \cite{Poly1}, \cite{Poly2} about ten years ago.
Namely, they proved that
\begin{align}
p_{n+1}-p_n \le 246\lb{hez01}
\end{align}
holds infinitely many times, where $p_n$ denotes the $n$-th prime number.

\bigskip It is also interesting to study
the maximum value of $p_{n+1}-p_n$ for primes less than $x$.
This problem was studied, for example, 
by Westzynthius \cite{West}, Erdős \cite{Erdos1}, and
Rankin \cite{Rankin}.
In 1997, \mbox{Erdős \cite{Erdos2} offered a reward
  of \$10,000}
to anyone who could improve on the previous best result,
but this was only achieved in 2016 by Ford, Green, Konyagin, Tao
\cite{Ford0}, and Maynard \cite{Maynard2}, \cite{Maynard3}
(also see the joint paper \cite{Ford2}).
%2. oldal
Here, they proved that
\[
p_{n+1}-p_n >\dfrac{c\ \log n\ \log \log n\ \log \log \log \log n}{\log
\log \log n}
\]
holds for infinitely many integers $n$, where $c > 0$ is an absolute constant.
Consider the following problem: is it possible to find
many large consecutive prime gaps?
Shiu \cite{Shiu}, Cram\'er \cite{Cramer}, Maier \cite{Maier},
and Pintz \cite{Pintz} also worked on this problem.
Finally, Ford, Maynard, and Tao \cite{Ford} proved the sharpest theorem:

\bigskip
\noindent\textbf{Theorem A.} [Ford, Maynard, Tao].
\textit{Let }$k\ge 1$\textit{ be an integer and} 
\[
G_k(x)\stackrel{=}{\textup{def}}\max_{p_{n+k}\le x} \min\{p_{n+1}-p_n,
p_{n+2}-p_{n+1},\dots,p_{n+k}-p_{n+k-1}
\}.
\]
\textit{Then}
\begin{align}
G_k(x) \gg \dfrac{1}{k^2}\dfrac{\log x\log\log x\log\log\log\log
  x}{\log\log\log x},\lb{g01}
\end{align}
\textit{where the implied constant is absolute and effective.}

\bs The following result, Theorem \ref{thm01}, provides a lower bound for $G_k(x)$ that, while staying within the average spacing of $\log x$, offers important insights for larger values of $k$. For the case $k=1$, the result is essentially a consequence of the Prime Number Theorem and could be considered trivial. However, for $k \ge 2$, the estimate is derived using specific bounds from sieve theory. This result sharpen Theorem A only for large $k$.  
  
\begin{thm}\lb{thm01}
  There exists a positive constant $x_0$ such that if $x_0<x\in\mathbb R$
  and $2\le k\in\mathbb N$, then
  \begin{align}
  G_k(x) > 0.1504\phantom{\cdot}\dfrac{\log x}{k}.\lb{g02} 
  \end{align}
\end{thm}

If $k$ is large in terms of $x$, i.e., 
\[
k\gg \dfrac{\log\log x\log\log\log\log x}{\log\log\log x},
\]
then the estimate \eqref{g02} is sharper than the estimate
\eqref{g01}.

% 3. oldal
\bigskip Today, according to \eqref{hez01}, we know that there are infinitely many consecutive primes whose difference is $\le 246$.
Related to this, I ask the following problem:

\begin{prob}\lb{prob01}
  Is it possible to prove the existence of an absolute constant $C$ for which
  there are infinitely many integers $n$ with
\[
p_{n+1}-p_n < C
\]
and 
\[
p_{n+1}-p_n
\]
is a square?
Is a similar statement true for squarefree numbers in place of squares?
Is such a statement true for primes in place of squares?
\end{prob}

I conjecture that the above problem is true for $C=246^2$,
but I have not been able to prove this. However, a slightly weaker statement
follows immediately from Pintz, Steiger, and Szemerédi theorem \cite{PStSz} 
or Bloom and Maynard's theorem \cite{bloom}.

\begin{thm}\lb{thm02}
  There are infinitely many prime pairs $p,q$ for which $p-q$ is a square.
\end{thm}

Problem \ref{prob01} is related to the famous
Dickson conjecture (see Conjecture \ref{conj01} later).
Surprisingly, if we replace the squares in Problem \ref{prob01}
with perfect powers,
then the statement can be easily proved, following the Banks, Freiberg,
and Turnage-Butterbaugh's theorem \cite{banks}.

\begin{thm}\lb{thm03}
  There are infinitely many integers $n$ for which
  \[
  \scaleto{p_{n+1}-p_n <}{10pt} \scaleto{\ e^{1.834\cdot 10^{76}}}{14pt}
  \]
  and $p_{n+1}-p_n$ is a perfect power. 
\end{thm}  

The method of the proof of Theorem \ref{thm03}
can also be used to prove the following more general statement:

%4. oldal
\begin{thm}\lb{thm04}
  For every integer $m\ge 1$, there are infinitely many integers $n$
  for which
  \[
  p_{n+m}-p_{n+1} < \scaleto{e^{e^{e^{10.6m}}}}{20pt}
  \]
  and the difference between any two of the consecutive primes
  $p_{n+1},p_{n+2},\dots,p_{n+m}$
  is a perfect power.
\end{thm}

\begin{prob}
  The upper bounds in Theorems \ref{thm03} and \ref{thm04}
  are clearly not optimal.
  It might be interesting to sharpen the constants in these theorems
  significantly.
\end{prob}

In addition to the cases concerning squares and powers, it is also interesting
to study sets of primes for which
the difference between two primes is never a square.
Then, using the results of Ruzsa \cite{Ruzsa} and Lewko \cite{Lewko},
I will prove the following:

\begin{thm}\lb{thm05} Let
  $\gamma = \dfrac{1}{2}+\dfrac{\log 12}{2\log 205}=0.7334117\dots$.
  If $x\ge 2$, then in the set $\{1,2,\dots,x\}$ there exist at
  least $\dfrac{1}{205}x^{\gamma}/\log x$ primes such that the
  difference between any two primes is never a square.
\end{thm}

We may also study the case when we require that the primes
in Theorem \ref{thm05} should be consecutive primes. In this case,
I have been able to prove the existence of a much smaller set of primes
with the desired property.

\begin{thm}\lb{thm06}
  There exists a positive constant $x_0$ such that if $x\ge x_0$,
  then in the set $\{1,2,\dots,x\}$ there exist $k=\left[0.24(\log x)^{1/4}\right]$
  consecutive primes such that the difference between any two of these primes is
  never a square.
\end{thm}

\begin{prob} Is Theorem \ref{thm06} true for $k=\left[x^c\right]$
  consecutive primes, where $c>0$ is an absolute constant? 
\end{prob}

The structure of this paper is as follows. In Section 2, we provide the proof of Theorem \ref{thm01} and discuss the underlying sieve methods. Section 3 is devoted to the proof of Theorems \ref{thm03} and \ref{thm04}, while the proofs of the remaining two theorems are detailed in Sections 4 and 5. To achieve these results, we utilize a combination of prime number theory and combinatorial methods. All numerical computations throughout the proofs were performed using the SAGE programming language \cite{SAGE}; the corresponding source codes are provided in the Appendix.

%5. oldal
\section{Proof of Theorem \ref{thm01}.}

Throughout the proof, we may assume that $k\le 0.0752\log x$;
otherwise, the theorem is trivial.
Let
\[
r=0.1504\phantom{\cdot}\dfrac{\log x}{k}\ge 2.
\]
Then,
\[
rk=0.1504\log x.
\]
Assume that $x$ is sufficiently large ($x>x_0$).
In the interval $[1,x]$, we indicate the primes with two colors, red and green.
The color of a prime $1< p\le x$ is red if there exists another prime
$q>p$ for which $q-p\le r$. The color of a prime $1< p\le x$ is
green if there does not exist such a prime $q$, i.e.,
for all $q>p$ primes, we have
\[
q-p > r.
\]
Let $R$ denote the number of red primes
and $G$ denote the number of green primes in the interval $[1,x]$.
The main idea of the proof is to find many consecutive green primes
in the interval $[1,x]$. To do this, we will prove the following:
\begin{lem}\lb{lem01}
  There exists a constant $x_0$ such that for all $x > x_0$, the number
  of red primes in $[1,x]$ satisfies
  \[
  R\le 6.646\phantom{\cdot}\dfrac{rx}{(\log x)^2}.
  \]
\end{lem}

Before proving the lemma, let us see how the theorem follows from it.
In the interval $[1,x)$, let us call a subinterval $[a,b)$ red if all the primes
it contains are red. Similarly, let us call a subinterval $[a,b)$ green if all the primes
it contains are green.
%6. oldal
The interval $[1,x)$ can be partitioned into the union of disjoint 
red and green intervals. Since these intervals alternate in colour 
(and we may assume that $[1,x)$ starts with a red interval), it follows 
that a red interval is followed by a green one, and a green interval is 
followed by a red one. In this partition, the number of green intervals 
is less than or equal to the number of red intervals, which in turn is 
less than or equal to the total number of red primes. That is,
\[
\textup{\#(red intervals)}\le R \le 6.646 \frac{rx}{(\log x)^2}.
\]
By Rosser and Schoenfeld's result \cite[Theorem 9]{Rosser}
\begin{equation}
\pi(x) >\dfrac{x}{\log x},\ \ \textup{ if }x\ge 17.\lb{primszamtetel}
\end{equation}  
Thus, for the number of green primes, we have
\[
G\ge \dfrac{x}{\log x} -R\ge \dfrac{x}{\log x}-6.646\dfrac{rx}{(\log x)^2}.
\]
That is, in the above division of $[1,x]$ into disjoint intervals,
we have at most $6.646\phantom{\cdot}\dfrac{rx}{(\log x)^2}$
green intervals, which contain more than
$\dfrac{x}{\log x} -6.646\phantom{\cdot}\dfrac{rx}{(\log x)^2}$
green primes. Then, by the pigeon-hole principle,
there exists a green interval that contains at least 
  \begin{align*}
    &\phantom{\ge}\dfrac{\dfrac{x}{\log x}
      -6.646\phantom{\cdot}\dfrac{rx}{(\log x)^2}}{6.646\phantom{\cdot}
      \dfrac{rx}{(\log x)^2}}
      \intertext{green primes, which is}
      &> 0.1504\phantom{\cdot}\dfrac{\log x}{r}-1\\
      &= 0.1504\phantom{\cdot}\dfrac{\log x}{0.1504 \log x/k}-1\\
      &=k-1.
  \end{align*}
  Denote the first $k$ of these green primes by
  $p_n,p_{n+1},\dots,p_{n+k-1}$. Since these primes are consecutive primes
  in a green interval, the distance between two adjacent ones is
  $> r =0.1504\phantom{\cdot}\dfrac{\log x}{k}$, from which
  Theorem \ref{thm01} follows.
  %7. oldal
  All that remains is to prove Lemma \ref{lem01}.
  
  \bigskip\noindent\textbf{The proof of Lemma \ref{lem01}.}
  A classical upper bound for the number of prime pairs with a given difference $h$ 
can be found in the work of Halberstam and Richert \cite[Theorem 3.11]{Halberstam}. 
While more recent developments (see e.g. \cite{Johnston, Lichtman}) 
provide sharper constants or fully explicit versions of this inequality, and 
explicit bounds holding for all $x > 1$ are also available in the literature 
(see e.g. Riesel and Vaughan \cite[Lemma 5]{Riesel}), the 
following asymptotic form is sufficient for our current purposes.

  \begin{lem}[{\cite[Theorem 3.11]{Halberstam}}]\lb{lem02}
    If $N\rightarrow\infty$ and $2\mid h$, then
    \begin{align*}
      &\phantom{=}\ab{\{p:\ p\le N,\ p+h=p' \textup{ and }
        p,p' \textup{ are primes } \}}\\
      &\le 8\prod_{2<p}\zj{1-\dfrac{1}{(p-1)^2}}
      \prod_{2<p\mid h} \dfrac{p-1}{p-2}
      \cdot\dfrac{N}{(\log N)^2}
      \zj{1+O\zj{\dfrac{\log\log N}{\log N}}},
    \end{align*}
    where the implied constant factor used in the
    $O$-term is absolute and does not depend on $h$.
  \end{lem}

In the following, we denote the integer part of $r$ by $[r]$.
If $h$ is odd, then $\ab{\{p:\ p\le N,\ p+h=p' \textup{ and }
  p,p'\textup{ are primes} \}}\le 1$, since by parity, we know that
one of the primes $p$ and $p'$ is even, so it is $2$.
Using this and Lemma \ref{lem02}, we find that the number of red primes is
  \begin{align}
    R &\le
    \sum_{\substack{h=2\\ 2\mid h}}^{[r]} 8
    \prod_{2<p}\zj{1-\dfrac{1}{(p-1)^2}}
    \prod_{2<p\mid h}\dfrac{p-1}{p-2}
    \cdot \dfrac{x}{(\log x)^2}
    \zj{1+O\zj{\dfrac{\log\log x}{\log x}}}\notag\\
    &+\sum_{\substack{h=1 \\ h\textup{ is odd}}}^{[r]} 1.\lb{becser}
  \end{align}
  In the above estimate, $O$ is independent of $h$ and
  $  \sum\limits_{\substack{h=1 \\ h\textup{ is odd}}}^{[r]} 1\le [r] <\log x$,
  thus, if
  $x$ is sufficiently large, then
  \[
  R < 8.0001
  \sum_{\substack{h=2\\ 2\mid h}}^{[r]}
   \prod_{2<p}\zj{1-\dfrac{1}{(p-1)^2}}
    \prod_{2<p\mid h}\dfrac{p-1}{p-2}
    \cdot \dfrac{x}{(\log x)^2}.
  \]
  A numerical computation using SAGE programming language
  (Appendix, Program 1) shows that
\[
\prod_{2<p}\zj{1-\dfrac{1}{(p-1)^2}}<\prod_{2<p<5000}
\zj{1-\dfrac{1}{(p-1)^2}} < 0.6602.
\]
These Euler products can be computed to high precision; for an extensive survey and numerical values, see e.g. Moree \cite{Moree}.
%8. oldal
Thus,
\begin{align}
R  &<5.2817
    \sum_{\substack{h=2\\ 2\mid h}}^{[r]} 
    \prod_{2<p\mid h}\dfrac{p-1}{p-2}
    \cdot \dfrac{x}{(\log x)^2}.\lb{thm01/01}
\end{align}
So, we need to give an upper bound for
\[
\sum_{\substack{h=2\\ 2\mid h}}^{[r]} 
    \prod_{2<p\mid h}\dfrac{p-1}{p-2}.
\]
Here,
\begin{align*}
  \dfrac{p-1}{p-2}
  &=\zj{1+\dfrac{1}{p}}\cdot\dfrac{(p-1)p}{(p-2)(p+1)}\\
  &=\zj{1+\dfrac{1}{p}}\zj{1+\dfrac{2}{(p-2)(p+1)}}\\
  &<\zj{1+\dfrac{1}{p}}e^{2/ ((p-2)(p+1))}.
\end{align*}  
Thus,
\begin{align}
 \sum_{\substack{h=2\\ 2\mid h}}^{[r]} 
 \prod_{2<p\mid h}\dfrac{p-1}{p-2}
 &<\sum_{\substack{h=2\\ 2\mid h}}^{[r]}
 \prod_{2<p\mid h} \zj{1+\dfrac{1}{p}}e^{2/ ((p-2)(p+1))}\notag\\
 &=\sum_{\substack{h=2\\ 2\mid h}}^{[r]}
 e^{\sum_{2<p\mid h} 2/((p-2)(p+1))}\prod_{2<p\mid h}\zj{1+\dfrac{1}{p}}.\lb{thm01/02}
\end{align}  
Let $T=20000$, then
\begin{align*}
  \sum_{2<p\mid h} &2/((p-2)(p+1))
  < \sum_{2<p} 2/((p-2)(p+1))\\
  &= \sum_{2<p\le T+1} 2/((p-2)(p+1))+ \sum_{T+2\le p} 2/((p-2)(p+1)).
\end{align*}
Using the SAGE programming language (Appendix, Program 2), we get
\[
\sum_{2<p\le T+1} 2/((p-2)(p+1)) < 0.7271.
\]
%9. oldal
Furthermore,
\begin{align*}
  \sum_{T+2\le p} 2/((p-2)(p+1))
  &<\sum_{n=T+2}^{\infty}  2/((n-2)(n+1))\\ 
  &< \sum_{n=T+2}^{\infty}  2/((n-2)(n-1))\\
  &< \sum_{n=T+2}^{\infty} \zj{\dfrac{2}{n-2}-\dfrac{2}{n-1}}\\
  &= \dfrac{2}{T}=0.0001.
\end{align*}
So,
\begin{align*}
  \sum_{2<p\mid h} 2/((p-2)(p+1))<0.7272.
\end{align*}  
Thus,
\begin{align*}
  e^{\sum_{2<p\mid h} 2/((p-2)(p+1))} <e^{0.7272}<2.0693.
\end{align*}  
Writing this in inequality \eqref{thm01/02}, we get
\begin{align}
  \sum_{\substack{h=2\\ 2\mid h}}^{[r]} 
  \prod_{2<p\mid h}\dfrac{p-1}{p-2}
  &<2.0693
  \sum_{\substack{h=2\\ 2\mid h}}^{[r]} \prod_{2<p\mid h} \zj{1+\dfrac{1}{p}}\notag\\
  &=2.0693
  \sum_{\substack{h=2\\ 2\mid h}}^{[r]} \prod_{p\mid h} \zj{1+\dfrac{1}{p}}
  /\zj{1+\dfrac{1}{2}}\notag\\
  &<1.3796 \sum_{\substack{h=2\\ 2\mid h}}^{[r]} \prod_{p\mid h}
  \zj{1+\dfrac{1}{p}}\notag\\
  &=1.3796 \sum_{\substack{h=2\\ 2\mid h}}^{[r]} \sum_{\substack{d\mid
      h\\ |\mu(d)|=1}}\frac{1}{d}\notag\\
  &=1.3796 \sum_{\substack{d=1\\ |\mu(d)|=1}}^{[r]} \dfrac{1}{d}
  \sum_{\substack{h=1\\ 2\mid h\\ d\mid h}}^{[r]} 1\notag\\
  &=1.3796\sum_{\substack{d=1\\ |\mu(d)|=1\\ d\textup{ even}}}^{[r]}
  \dfrac{1}{d}\sum_{\substack{h=1\\ d\mid h}}^{[r]} 1
  +1.3796\sum_{\substack{d=1\\ |\mu(d)|=1\\ d\textup{ odd}}}^{[r]}
  \dfrac{1}{d}\sum_{\substack{h=1\\ 2d\mid h}}^{[r]} 1\notag\\
  &\le 1.3796\sum_{\substack{d=1\\ |\mu(d)|=1\\ d\textup{ even}}}^{[r]}
  \dfrac{r}{d^2}
  + 1.3796\sum_{\substack{d=1\\ |\mu(d)|=1\\ d\textup{ odd}}}^{[r]}
  \dfrac{r}{2d^2}\notag\\
  &< 1.3796r\sum_{\substack{d=1\\ |\mu(d)|=1\\ d\textup{ even}}}^{\infty}
  \dfrac{1}{d^2}
  + 1.3796r\sum_{\substack{d=1\\ |\mu(d)|=1\\ d\textup{ odd}}}^{\infty}
  \dfrac{1}{2d^2}\lb{thm01/03}
\end{align}
%10. oldal
Let $T=5000$. Using the SAGE programming language (Appendix Programs 3 and 4)
we get
\begin{align*}
  \sum_{\substack{d=2\\ |\mu(d)|=1\\ d\textup{ even}}}^{T}
  \dfrac{1}{d^2}
  &=\sum_{\substack{d=1\\ d\textup{ even}}}^{T}
  \dfrac{\mu^2(d)}{d^2}<0.304   \\
  \sum_{\substack{d=1\\ |\mu(d)|=1\\ d\textup{ odd}}}^{T+1}
  \dfrac{1}{2d^2}
  &=\sum_{\substack{d=1\\ d\textup{ odd}}}^{T+1}
  \dfrac{\mu^2(d)}{d^2}<0.6079.\\
\end{align*}  
Moreover, 
\begin{align*}
  \sum_{\substack{d=T+2\\ |\mu(d)|=1\\ d\textup{ even}}}^{\infty}
  \dfrac{1}{d^2}&<\sum_{\substack{d=T+2\\ d\textup{ even}}}^{\infty}
  \dfrac{1}{d(d-2)}
  <\sum_{\substack{d=T+2\\ d\textup{ even}}}^{\infty}
  \dfrac{1}{2}\phantom{\cdot}\zj{\dfrac{1}{d-2}-\dfrac{1}{d}}\\
  &=\dfrac{1}{2T}=0.0001\\
  \intertext{and}
  \sum_{\substack{d=T+3\\ |\mu(d)|=1\\ d\textup{ odd}}}^{\infty}
  \dfrac{1}{2d^2}
  &<\sum_{\substack{d=T+3\\ d\textup{ odd}}}^{\infty}
  \dfrac{1}{2d(d-2)}
  <\sum_{\substack{d=T+3\\ d\textup{ odd}}}^{\infty}
  \dfrac{1}{4}\phantom{\cdot}\zj{\dfrac{1}{d-2}-\dfrac{1}{d}}\\
  &< \dfrac{1}{4T}=0.00005.
\end{align*}
% 11. oldal
Thus, by \eqref{thm01/03} we get
\begin{align*}
  \sum_{\substack{h=2\\ 2\mid h}}^{[r]} 
  \prod_{2<p\mid h}\dfrac{p-1}{p-2}
  &<1.3796r\zj{0.3041+0.60795}<1.2583r. 
\end{align*}
Writing this in \eqref{thm01/01}, we have
\[
R< 6.646 \phantom{\cdot} \dfrac{rx}{(\log x)^2},
\]
which proves  Lemma \ref{lem01}. This completes the proof of
Theorem \ref{thm01}.\qed

\bigskip \noindent\textbf{Remark.}
In Theorem \ref{thm01}, the constant factor could be further sharpened by employing more sophisticated sieve estimates. For instance, one could use the classical result of Chen \cite{Chen} or the more recent refinements by Johnston \cite{Johnston} and Lichtman \cite{Lichtman}. However, even with these high-precision improvements, the constant 0.1504 in Theorem \ref{thm01} would improve only by a few thousandths, which does not qualitatively change the nature of the result.

\section{Proof of Theorem \ref{thm02}.}

We can prove that for every integer $M$, there exist two distinct
primes greater than $M$ ($p$ and $q$), such that $p-q$ is a square number.
Throughout the proof, we fix $M$. Next, we will use the
theorem of Bloom and Maynard \cite[Theorem 1]{bloom}, who proved the
following:

\begin{lem}[Bloom-Maynard]\lb{lem03}
  There exists positive constants $N_0,c_0$ and $c_1$
  such that if $N\ge N_0$,
  $\mc A\subseteq\{1,2,\dots,N\}$ and
  \[
  \ab{\mc A} \ge c_0 \dfrac{N}{(\log N)^{c_1\log\log\log N}},
  \]
  then $\mc A$ contains two different elements whose difference
  is a square.
\end{lem}

It is important to emphasize that Lemma \ref{lem03} could be replaced 
by the density theorem of Pintz, Steiger, and Szemerédi \cite{PStSz}. 
Their result was the first to provide a sufficiently strong upper bound 
for the size of a set lacking square differences to be applicable in 
the context of prime gaps. In the following proof, we use the Bloom--Maynard 
formulation for its convenient explicit constants, but the underlying 
arithmetic conclusion remains consistent with the Pintz--Steiger--Szemerédi 
approach.

Next, we choose constants $N_0,c_0, c_1$ which satisfy Lemma \ref{lem03}.
Moreover, let $N\ge N_0$ and $\mc P$ be the set of primes between $M+1$ and $N$.
Then, if $N$ is sufficiently large depending on $M$ and $c_0,c_1$,
we can apply the prime number theorem:
\[
\ab{\mc P}
\ge 0.9\dfrac{N}{\log N}  \ge c_0 \dfrac{N}{(\log N)^{c_1\log\log\log N}}.
\]
So, using Lemma \ref{lem03}, we proved that there exist
two distinct elements $p\in\mc P$ and $q\in\mc P$,
for which $p-q$ is a square number. Then, by the definition of $\mc P$,
$p,q$ are primes and $M< q<p$, which completes the proof of the statement.
\qed

%12. oldal
\section{Proofs of Theorems \ref{thm03} and \ref{thm04}.}

The definition of the admissible set was proposed during the
development of the well-known Dickson conjecture,
which is as follows.

\begin{dfn}
We say that the set
\{$g_1x+h_1,g_2x+h_2,\dots g_kx+h_k\}\subseteq \mathbb Z[x]$
is admissible if
\[
\prod_{i=1}^{k} (g_ix+h_i)
\]
has no fixed prime divisor $p$, i.e.,
\[
\ab{\{n\pmod{p}:\ \ \prod_{i=1}^{k} (g_in+h_i)\equiv 0\pmod{p}\}}<p
\]
holds for all primes $p$.
\end{dfn}
If we consider the case $g=1$, this definition states that the 
set $\{x+h_1,x+h_2,\dots,x+h_k\}\subset \mathbb Z[x]$ is admissible 
if and only if for every prime $p$, the set $\{h_1,h_2,\dots,h_k\}$ 
does not include a complete residue system modulo $p$.

\begin{conj}[Dickson's conjecture]\lb{conj01}
  If the set $\{g_1x+h_1,g_2x+h_2,\dots g_kx+h_k\}\subseteq \mathbb Z[x]$
  is admissible, then there are infinitely many positive integers $n$
  for which every element of the set $\{g_1n+h_1,g_2n+h_2,\dots g_kn+h_k\}$
  is prime.
\end{conj}

Dickson's conjecture remains unsolved; however, one of the most significant 
breakthroughs towards this open problem is attributed to Maynard 
\cite{Maynard1} and Tao. Their work established the existence of bounded 
gaps between primes for any $k$-tuple, significantly advancing the earlier 
pioneering results in this direction. The following formulation of their 
result was given by Granville \cite[Theorem 6.4.]{Granville}.

\begin{lem}[Maynard-Tao]\lb{lem04}
  For every $m\in \mathbb N$, there exists a number $k_m$
  depending only on $m$ such that if $k\ge k_m$ and the set
  $\{g_1x+h_1,\dots,g_kx+h_k\}$ is admissible, then there exist
  infinitely many positive integers $n$ for which the
  set $\{g_1n+h_1,\dots,g_kn+h_k\}$ contains at least $m$ primes.
\end{lem}

\begin{dfn}
  Denote the smallest constant $k_m$ for which Lemma \ref{lem04}
  holds by $K_m$.  
\end{dfn}  
%13. oldal
Maynard \cite{Maynard1} and Tao also gave an upper bound for $K_m$;
namely, they proved that
\begin{equation}
K_m\log K_m < e^{8m+4}. \lb{km01}
\end{equation}
This inequality follows from the methods established by Maynard and Tao
\cite{Maynard1};
however, for a version that provides this explicit form, we refer to the exposition by Granville \cite[Theorem 6.4]{Granville}.

Later, Maynard \cite[Propositions 4.2. and 4.3.]{Maynard1}
proved that $K_2\le 105$ and $K_m\le cm^2e^{4m}$.
The sharpest upper bounds were proved in the
framework of the Polymath Project \cite[Theorem 3.2]{Poly2}:
\begin{align}
  K_2&\le 50, \lb{km02}\\
  K_m&\le ce^{(4-28/157)m} \lb{km03}.
\end{align}  

In order to prove Theorems 3 and 4, I will use estimates
\eqref{km01} and \eqref{km02}. The main tool in the proofs is the
theorem of Banks, Freiberg, and Turnage-Butterbaugh \cite{banks}.

\begin{lem}[Banks, Freiberg, Turnage-Butterbaugh]\lb{lem05}
  Let $m,k\in\mathbb N$, $m\ge 2$, $k\ge K_m$. Furthermore,
  let $\{x+b_1,\dots, x+b_k\}\subset\mathbb Z[x]$ be an admissible set
  and $g\in\mathbb N$ such that $\textup{gcd}(g,b_1\cdots b_k)=1$.
  Then, there is a subset $\{h_1,\dots,h_m\}\subset\{b_1,\dots,b_k\}$
  such that there exist infinitely many $r\in\mathbb N$ for which
  the elements of $\{gr+h_1,\dots,gr+h_m\}$ are consecutive primes.
\end{lem}

We will use this lemma for a properly constructed set
$\{b_1,\dots,b_{K_m}\}$. This set will be of the form
$\{Wa,2Wa,\dots,K_mWa\}$, where $a$ and
$W$ are defined in the following lemma.

\begin{lem}\lb{lem06}
  Let $m\ge 2$ be an integer. Define $W$ by
  \begin{equation}
  W=\prod_{\substack{p\le K_m\\ \textup{p prime}}}p\lb{W}.
  \end{equation}
  Then, there exist an $0<a\in\mathbb N$,
  for which all elements of
  \[
  \{Wa,2Wa,\dots,(K_m-1)Wa\}
  \]
  are perfect power. Furthermore, we also have 
  \begin{align}
    a &< \scaleto{e^{e^{e^{10.5m}}}}{20pt}, \lb{lem06a}
    \intertext{and, in the special case of $m=2$ we also have}
    a &< \scaleto{10^{1.8339\cdot 10^{76}}}{14pt}.\lb{lem06b}
  \end{align}
\end{lem}

%14. oldal
\bigskip\noindent\textbf{Proof of Lemma \ref{lem06}.}
First we prove \eqref{lem06a}. Let $s$ be the largest integer such that 
$p_s \le K_m$, where $p_1, p_2, \dots, p_s$ denote the sequence of primes 
in increasing order. In other words, $p_1, \dots, p_s$ lists all the 
primes in the interval $[2, K_m]$.
Then, $W=p_1p_2\cdots p_s$. We also need the
first $K_m$ primes, which are $p_1=2,p_2=3,\dots,p_{K_m}$.
Next, we will find an integer $a$, which is of the form
\[
a=p_1^{\alpha_1}\cdots p_s^{\alpha_s},
\]
$a$ is as small as possible, and
for $1\le i\le K_m$, $iWa$ is a perfect $p_i$-th power
(i.e., it is of the form $iWa=x^{p_i}$, where $0<x\in\mathbb N$).
This determines the remainder of the numbers
$\alpha_1,\alpha_2,\dots,\alpha_s$ modulo $p_i$ for each $1\le i \le K_m$.
That is, for a fixed $j$, the remainder of an $\alpha_j$ is given
for each of the moduli $p_1,p_2,\dots,p_{K_m}$.
Thus, $\alpha_j$ satisfies a simultaneous congruence system,
where the moduli $p_1,p_2,\dots,p_{K_m}$
are pairwise relative primes. According to
the Chinese remainder theorem, such an $\alpha_j$ exists, and
for the smallest $\alpha_j$ with this property, we have
\[
0\le \alpha_j\le p_1p_2\cdots p_{K_m}.
\]
Thus,
\[
a=p_1^{\alpha_1}\cdots p_s^{\alpha_s}
\le (p_1\dots p_s)^{p_1p_2\dots p_{K_m}}.
\]
According to Rosser and Schoenfeld's result \cite[Theorem 9]{Rosser},
$\prod_{p\le x} p< e^{1.01624x}$, so
\begin{align}
  W &=\prod_{p\le K_m}p=p_1p_2\cdots p_s
  < e^{1.01624K_m}\le \scaleto{e^{1.01624e^{8m+4}}}{14pt}.\lb{w01}
\end{align}
We note that while this bound has been significantly improved in recent years (see e.g., \cite{Bordignon}), the classical bound is sufficient for our purposes.
On the other hand, according to Dusart's result \cite[Lemma 1]{Dusart},
$p_k\le k(\log k+\log\log k)$, if $k\ge 6$,
we can also estimate $a$. Let $K=\max\{K_m,6\}$. Then,
\begin{align}
  a &\le \scaleto{\zj{e^{1.01624K_m}}^{p_1\dots p_{K_m}}}{18pt}\notag\\
  &=\scaleto{e^{1.01624K_m{p_1\dots p_{K_m}}}}{14pt}\notag\\
  &<\scaleto{e^{1.01624K\cdot e^{1.01624K(\log K+\log\log K)}}}{16pt}.\notag
\end{align}
%15. oldal
Here, by \eqref{km01} we have $K\log K<e^{8m+4}$, so
$K<e^{8m+4}$ and $\log K+\log\log K < 8m+4$.
Thus, 
\begin{align}
  a &< \scaleto{e^{1.01624e^{8m+4}e^{1.01624 e^{8m+4}(8m+4)}}}{20pt}\lb{becs01}\\
    &< \scaleto{e^{e^{e^{10.5m}}}}{20pt},\ \ \ \textup{if }m\ge 3.\notag 
\end{align}
Note that for $m=2$ the upper estimate in \eqref{becs01} is approximately
\[
\scaleto{10^{2.0407 \cdot 10^{4282528195}}}{14pt},
\]
which is an order of magnitude worse than the upper estimate given
in \eqref{lem06b}.

Next, we prove \eqref{lem06b}. This requires slightly more precise
calculations, for which I used the SAGE programming language again.
In the proof, we use the fact that $K_2\le 50$ (see \eqref{km02}). Let us list the primes
between $1$ and $50$: $p_1=2,p_2=3,\dots,p_{15}=47$.
Then, $W\le p_1\cdots p_{15}$.
Now, we will find an integer $a$ of the form
\begin{equation}
a=p_1^{\alpha_1}p_2^{\alpha_2}\cdots p_{15}^{\alpha_{15}},\lb{acska}
\end{equation}
such that for every $1\le i\le 49$, $iWa$ is a perfect
power. To do this, we divide the set $\{1,2,\dots,49\}$ into three parts:
\begin{align*}
  \mc H_1&=\{1,4,9,16,25,49\},\\
  \mc H_2&=\{3,24\},\\
  \mc H_3&=\{1,2,\dots,49\}\setminus(\mc H_1\cup\mc H_2).
\end{align*}  
Then, $|\mc H_3|=41$. Let us define a bijection $f$ between the elements
of $\mc H_3$ and the prime numbers $\{p_3=5,p_4=7,p_5=11,\dots,p_{43}\}$.
We will construct a positive integer $a$ for which:
\begin{description}
\item $Wa$ is a square,
\item $3Wa$ is a cube,
\item and for $i\in\mc H_3$, $iWa$ is an $f(i)$-th power, 
  (i.e., it is of the form $iWa=x^{f(i)}$, where $0<x\in\mathbb N$ and
  $f(i)$ is a prime in the set $\{p_3,p_4,\dots,p_{41}\}$).
\end{description}
% 16. oldal
Then, for every $i\in \{1,2,\dots,49\}$, $iWa$ is a perfect power
since for $i\in\mc H_1$, $iWa$ is a square, for $i\in\mc H_2$,
$iWa$ is a cube, and
for $i\in\mc H_3$, $iWa$ is an $f(i)$-th power. As before, now, each
exponent $\alpha_j$ in the prime factorization of the integer $a$
given in \eqref{acska} must satisfy a
simultaneous congruence system, where the moduli are the first 43 primes.
Since these moduli are also pairwise relative primes, we can use
the Chinese remainder theorem, and get
\[
\alpha_j< \prod_{j=1}^{43} p_j\stackrel{\textrm{def}}{=}E.
\]
Then,
\[
a\le (p_1\cdots p_{15})^{E}.
\]
I calculated this value using the SAGE programming language
(Appendix, Program 5), and got the following result:
\begin{align}
W&\le p_1p_2\cdots p_{15}<10^{17,789}\lb{W2}\\  
a&<\scaleto{10^{1.8339\cdot 10^{76}}}{14pt}.\notag
\end{align}
This completes the proof of Lemma \ref{lem06}.\qed

\bigskip Now, we can prove Theorems \ref{thm03} and \ref{thm04}.
In the case of the proof of Theorem \ref{thm03}, let $m=2$.
In Theorems \ref{thm03} and \ref{thm04}, we look for an
infinite number of consecutive primes $p_{n+1},p_{n+2},\dots,p_{n+m}$
such that for $1\le i<j\le m$, $p_{n+j}-p_{n+i}$ is a perfect power, and
\begin{align*}
p_{n+m}-p_{n+1} &< \scaleto{e^{e^{e^{10.6m}}}}{20pt},\ \ \textup{if }m\ge 3,
\intertext{respectively,}
p_{n+2}-p_{n+1} &<\scaleto{10^{1.834\cdot 10^{76}}}{14pt} , \ \ \textup{if }m=2.
\end{align*}

For this, consider the $a$ and $W$ given in Lemma \ref{lem06}.
Then, every element of the set $\{Wa,2Wa,\dots,(K_m-1)Wa\}$ is a
perfect power. Next, we would like to apply Lemma \ref{lem05} 
with $k=K_m$, $\{b_1,b_2,\dots,b_{k}\}=\{Wa,2Wa,\dots,K_mWa\}$ and $g=1$.
% 17. oldal
Then, the set $\{x+b_1,\dots,x+b_{K_m}\}$ is admissible
since $\{Wa,2Wa,\dots,K_mWa\}$ does not contain a complete residue system
for any prime $p$. This is because if $p\le K_m$, then by $p\mid W$,
every element of the set is congruent with $0$ modulo $p$, and if $p>K_m$,
then the set has fewer elements than $p$. By Lemma \ref{lem05},
there is a subset $\{h_1,h_2,\dots,h_m\}\subset \{Wa,2Wa,\dots,K_mWa\}$
for which there exists infinitely many $r$ such that $r+h_1,r+h_2,\dots,r+h_m$
are consecutive primes. Let
\[
p_{n+1}=r+h_1,\dots,p_{n+m}=r+h_m .
\]
Then, for $1\le i< <j\le m$,
\[
p_{n+j}-p_{n+i}=h_j-h_i\in\{Wa,2Wa,\dots,(K_m-1)Wa\},
\]
such that $p_{n+j}-p_{n+i}$ is a perfect power.
On the other hand, for the $a$ given in Lemma \ref{lem05},
we have \eqref{lem06a}
or \eqref{lem06b}, and for $W$, we have \eqref{w01} or \eqref{W2}; thus,
\begin{align*}
  p_{n+m}-p_{n+1} &\le (K_m-1)Wa
  <  K_mW\scaleto{e^{e^{e^{10.5m}}}}{20pt}\\
  &<e^{8m+4}e^{1.01624e^{8m+4}}\scaleto{e^{e^{e^{10.5m}}}}{20pt}
  <\scaleto{e^{e^{e^{10.6m}}}}{20pt},\ \ \textup{if }m\ge 3, 
  \intertext{and for $m=2$, we have}
    p_{n+2}-p_{n+1} & \le (K_2-1)Wa
  < K_2W\scaleto{10^{1.8339\cdot 10^{76}}}{14pt}\\
  &< 50\cdot{10^{17.781}} \scaleto{10^{1.8339\cdot 10^{76}}}{14pt}
  <\scaleto{10^{1.834\cdot 10^{76}}}{14pt}.
\end{align*}
This completes the proofs of Theorems \ref{thm03} and \ref{thm04}.\qed

\bigskip\noindent\textbf{Remark.} If we could find $50$ distinct
squares such that the difference between any two is also a square, then
Problem \ref{prob01} could be proved using the method described here.
However, I strongly doubt that such a set of $50$ elements exists.
In fact, I do not think there is even a set of five elements with
the above property (see Conjecture 4 in \cite{Gyarmati}).
It is also worth mentioning that based on the Bombieri-Lang conjecture,
an absolute constant upper bound can be given for the size of sets of
integers where the difference between any two numbers is a square
(for example, see the Note after Theorem 4.6 in \cite{Alon}).
Unfortunately, we do not know whether this constant is
greater than or less than $50$.

% 18. oldal
\section{Proof of Theorem \ref{thm05}.}

We can assume that $x\ge 17$ since the theorem is trivial
for $2\le x\le 16$. Let $D(x)$ denote the number of elements of the
largest set in $[1,x]$ that do not have two elements whose difference is
a square, and let $D_P(x)$ denote the number of elements of the largest
set of primes in $[1,x]$ that do not have two elements whose difference is
a square. Ruzsa \cite{Ruzsa}, in disproving a conjecture of Erdős,
proved the following:

\bigskip\noindent\textbf{Theorem B.}
\[
D(x)>\frac{1}{65}x^{\gamma},
\]
\textit{where}
\[
\gamma=\frac{1}{2}\zj{1+\frac{\log 7}{\log 65}}= 0.733077\dots\ .
\]

\bigskip\noindent In fact, Ruzsa proved a slightly more general theorem.
Namely, if $r(m)$ denotes the maximum number of residue classes modulo $m$
such that the difference between any two residue classes is not a square
modulo $m$, then the following holds.

\bigskip\noindent\textbf{Theorem C.}
  \textit{For all squarefree $m$ we have}
  \[
  D(x)\ge\frac{1}{m}x^{\gamma (m)},
  \]
  \textit{where}
  \[
  \gamma (m) =\frac{1}{2}+\frac{\log r(m)}{2\log m}.
  \]

\noindent Ruzsa noticed that
\begin{equation}
r(65)\ge 7.\lb{r65}
\end{equation}
Namely, if we consider the pairs
\[
(0, 0),\ (0, 2),\ (1, 8),\ (2, 1),\ (2, 3),\ (3, 9),\ (4, 7)
\]
where we assign to the pairs of numbers the residue class modulo $65$,
whose residue modulo $5$ is the first member of the pair,
while the residue $13$ is the second member of the pair,
then for these residue classes mod $65$, there are no two elements in the
set whose difference is a square modulo $65$, since the difference between any
two elements mod $5$ or mod $13$ is a quadratic non-residue.
% 19. oldal
This proves \eqref{r65}. Later, Lewko \cite{Lewko} gave a similar
construction with modulus $205$, proving that
\[
r(205)\ge 12.
\]
The only question is whether the proof of Theorem C can be transferred
to the case where the set can only contain prime numbers.
We will prove the following.

\begin{lem}\lb{lem07}
  For all $x\ge 17$ and squarefree $m$ we have
  \[
  D_P(x)\ge\frac{x^{\gamma(m)}}{m\log x},
  \]
  \textit{where}
  \[
  \gamma (m) =\frac{1}{2}+\frac{\log r(m)}{2\log m}.
  \]
\end{lem}  
Using Lemma \ref{lem07} with $m=205$ and $r(205)\ge 12$,
we immediately obtain the statement of Theorem \ref{thm05}.
Next, the proof of Lemma \ref{lem07} follows.

\bigskip\noindent\textbf{Proof of Lemma \ref{lem07}.}
Let $R\subset \mathbb Z_m$ be a set such that the difference between any two
elements is not a square modulo $m$ and $|R|=r(m)$.
If $r\in\mathbb Z_m$, then let $\overline{r}$ denote
the smallest nonnegative integer that represents $r$ among integers, i.e.,
\[
r\equiv\overline{r}\pmod{m},\ \ \ 0\le \overline{r}<m.
\]
If $0\le a<m$, define $R_a$ by the following set:
\[
  R_a\stackrel{\textup{def}}{=}
  \{\overline{r+a}:\ r\in R\}\subset \{0,1,2,\dots,m-1\}.
\]
Then, the difference between two elements in $R_a$ is never a square
modulo $m$, since the difference between elements in $R_a$ modulo $m$
is the same as the difference between the corresponding elements in $R$.

% 20. oldal
Let $n = \lfloor \log_m x \rfloor$, so that
\begin{equation}
m^n\le x<m^{n+1}. \label{5mn}
\end{equation}
We consider the $k$-tuples $(i_0, i_2, \dots, i_{t})$, where $t$ is the largest even integer less than $n$, and $k = \lfloor n/2 \rfloor + 1$ denotes the number of coordinates. For each such tuple, where $0 \le i_{2j} < m$ are integers, we define the set $S(i_0, i_2, \dots, i_t)$ as the collection of integers that can be written in the form
\[
s = \sum_{j=0}^{n-1} r_j m^j + 1,
\]
where $r_j \in R_{i_j}$ if $j$ is even, and $0 \le r_j < m$ is an arbitrary integer if $j$ is odd.
We will prove that the difference between any two elements in $S(i_0, i_2, \dots, i_t)$ cannot be a square.
Suppose that $s-s'=u^2$ , where $s,s'\in S(i_0,i_2,\dots,i_t)$. Let
\[
s=\sum_{j=0}^{n-1} r_j m^j+1,\ \ \ s'=\sum_{j=0}^{n-1} r_j' m^j+1.
\]
Let $k$ denote the smallest index $k$ for which
$r_k\ne r_k'$. Then:
\[
u^2=s-s'=(r_k-r_k')m^k+zm^{k+1}.
\]
If $k$ is odd, then $m^k\mid u^2$, but $m^{k+1}\nmid u^2$, which is
impossible for squarefree $m$.
If $k$ is even, let $k=2\ell$, and then
\[
\zj{u/m^{\ell}}^2\equiv r_k-r_k'\pmod{m}, \ \ \ \ \textup{where }r_k,\ r_k'
\in R_{i_k},
\]
which contradicts the fact that the difference between any
two elements in $R_{i_k}$ is not a square modulo $m$.

For every integer $1\le s\le m^n$, $s-1$
can be written in the $m$ number system:
\[
s=(s-1)+1=\sum_{j=0}^{n-1} r_jm^j+1.
\]
% 21. oldal
For a fixed even $j$, there exist exactly $r(m)$ values of $i$ such that $0 \le i < m$ and $r_j \in R_i$.
That is, for a given integer $1\le s\le m^n$, there exist a total of
$r(m)^{[(n+1)/2]}$ pieces of $[(n+1)/2]$-tuples
$(i_0,i_2,i_4,\dots,i_t)$ such that
\[
s\in S(i_0,i_2,\dots,i_t).
\]
This is true for all integers $1\le s\le m^n$, including the primes.
Thus, 
\begin{align}
  |\left\{ (p, (i_0,i_2,\dots,i_t))\right.
  &:\left. \ 1\le p\le m^n
  \textup{ prime,}\   p\in S(i_0,\dots,i_t),\ 0\le i_j<m,\right\}|
  \lb{szamn}\\
&=r(m)^{[(n+1)/2]} \pi(m^n).\notag
\end{align}
Since the number of $k$-tuples $(i_0,i_2,\dots,i_t)$ with $0 \le i_j < m$ is exactly $m^k$, according to the pigeon-hole principle, there exists a $k$-tuple $(i_0,i_2,\dots,i_t)$ such that $S(i_0,i_2,\dots,i_t)$ contains at least
\[
\dfrac{r(m)^k \pi(m^n)}{m^k}
\]
primes.
If we prove that
\begin{align}
\dfrac{r(m)^{[(n+1)/2]}\pi (m^n)}{m^{[(n+1)/2]}}\ge
\dfrac{x^{\gamma(m)}}{m\log x},\lb{5ut}
\end{align}
then the number of primes in this set $S(i_0,i_2,\dots,i_t)$ is $\ge \dfrac{x^{\gamma(m)}}{m\log x}$, and the difference between any
two primes is not a square number, which proves the
statement of Lemma \ref{lem07}. So, we only need to prove \eqref{5ut}.
In \cite[Corollary 1]{Rosser}, Rosser and Schoenfeld
proved that
for $x\ge 17$, $\pi(x) > \dfrac{x}{\log x}$. So,
\begin{align*}
  \dfrac{r(m)^{[(n+1)/2]}\pi (m^n)}{m^{[(n+1)/2]}}
  &\ge \dfrac{r(m)^{[(n+1)/2]}m^n}{m^{[(n+1)/2]}\log (m^n)}.
  \intertext{By \eqref{5mn},}
  \dfrac{r(m)^{[(n+1)/2]}\pi (m^n)}{m^{[(n+1)/2]}}
  &\ge \dfrac{r(m)^{[(n+1)/2]}m^n}{m^{[(n+1)/2]}\log x}\\
  &=\dfrac{1}{\log x} m^{[(n+1)/2]\zj{\log r(m)/ \log m -1}+n}.  
    \intertext{Since by $\log r(m)/ \log m -1$ is negative,}   
  \dfrac{r(m)^{[(n+1)/2]}\pi (m^n)}{m^{[(n+1)/2]}}&\ge \dfrac{1}{\log x}
  m^{ (n+1)\zj{\log r(m)/ \log m -1}/2+n}\\
  &=\dfrac{1}{m\log x}
    m^{(n+1)\log r(m)/(2\log m)+(n+1)/2}.
 %22. oldal   
    \intertext{Thus, by \eqref{5mn},}
    \dfrac{r(m)^{[(n+1)/2]}\pi (m^n)}{m^{[(n+1)/2]}}
    &\ge\dfrac{1}{m\log x}  x^{\log r(m)/(2\log m)+1/2}\\   
  &=\dfrac{x^{\gamma (m)}}{m\log x},
\end{align*}
which proves \eqref{5ut}. This completes the proof of the lemma. \qed

\section{Proof of Theorem \ref{thm06}.}

Let us list the primes between $2$ and $x$ in increasing order:
\[
2=p_1<p_2<p_3<\dots<p_s\le x.
\]
Let
\begin{align}
  t &=[0.24 (\log x)^{1/4}] \lb{tdef},                \\ 
  r &= \dfrac{t^{1/3}(\log x)^{2/3}}{3.323^{1/3}}.\lb{rdef}  
\end{align}
We color the primes in the interval $[1,x]$ with three colors. The color
of the prime $2\le p_i\le x$ is red if there exists an integer
$1\le h\le r$ such that $p_i+h^2$ is also prime. The color of
the prime $2\le p_i\le x$ is yellow if
\[
p_{i+t}-p_i> r^2.
\]
The primes that are not colored red or yellow are colored
green. In the proof, we look for $t$ consecutive green primes.
If $p_{i},p_{i+1},\dots,p_{i+t-1}$ are $t$ consecutive green primes,
then for $0\le m< n\le t-1$,
\[
p_{i+m}-p_{i+n}
\]
is not a square. This is because if
%23. oldal
\[
p_{i+m}-p_{i+n}=h^2
\]
holds, then for $h\le r$, the color of $p_{i+n}$ is red, and if $h>r$, then
by $p_{i+t}-p_i>p_{i+m}-p_{i+n}=h^2>r^2$, the color of $p_i$ is yellow.
Next, we will use the following lemma:

\begin{lem}\label{lem08}
  For sufficiently large $x$, denote the number of red primes in the interval $[1,x]$ by $R$, and the number of yellow primes by $Y$. Then
  \begin{align}
    R &\le 6.646 \dfrac{rx}{(\log x)^2}, \label{lem08a}\\
    Y &\le \dfrac{tx}{r^2} \label{lem08b}.
  \end{align}  
\end{lem}
Before proving the lemma, let us see how the theorem follows from it. We divide the interval $[1,x)$ into disjoint subintervals of the form $[a,b)$. We call such a subinterval \textit{red} if all primes within it are either yellow or red; otherwise, if a subinterval contains only green primes, it is called a \textit{green} interval.

The interval $[1,x)$ can be represented as a union of alternating disjoint red and green intervals. By our construction, the first interval starting from 1 is red (since the initial small primes do not satisfy the conditions for being green). Because red and green intervals alternate, the number of green intervals is at most the number of red intervals. Note that each red interval contains at least one red or yellow prime by definition. Thus, the total number of red and yellow primes, $R+Y$, provides an upper bound for the number of such intervals:
\[
R+Y \le 6.646 \frac{rx}{(\log x)^2} + \frac{tx}{r^2}.
\]
Using the prime number theorem in the form $\pi(x) > x/\log x$, the number of green primes $G$ satisfies:
\[
G \ge \pi(x) - (R+Y) \ge \frac{x}{\log x} - 6.646 \frac{rx}{(\log x)^2} - \frac{tx}{r^2}.
\]
% 24. oldal
Thus, we have at most
$6.646\phantom{\cdot}\dfrac{rx}{(\log x)^2}+\dfrac{tx}{r^2}$
green intervals, which contain more than
$\dfrac{x}{\log x}-6.646\dfrac{rx}{(\log x)^2}-\dfrac{tx}{r^2}$
green primes. So, by the pigeon-hole principle,
there exists a green interval containing at least
  \begin{align*}
    H\stackrel{\textup{def}}{=}\dfrac{\dfrac{x}{\log x}-6.646
      \dfrac{rx}{(\log x)^2}-\dfrac{tx}{r^2}}{6.646
      \phantom{\cdot}\dfrac{rx}{(\log x)^2}+\dfrac{tx}{r^2}}
  \end{align*}
  If this quantity is greater than $0.24 (\log x)^{1/4}$,
  the statement of the theorem is proved. Indeed,  
  \begin{align*}
    H=\dfrac{\dfrac{x}{\log x}}{6.646\cdot\dfrac{rx}{(\log x)^2}
    +\dfrac{tx}{r^2}}-1.
  \end{align*}
  Then, according to \eqref{rdef}, given as the definition of $r$
  \begin{align*}
    H=\dfrac{(\log x)^{1/3}}{3\cdot 3.323^{2/3} t^{1/3}}-1,
  \end{align*}  
  where by \eqref{tdef}, we have
  \begin{align*}
    H > 0.2408 \cdot (\log x)^{1/4},
  \end{align*}
   and this was to be proved. All that remains is to prove Lemma \ref{lem08}.
  
\bigskip\noindent\textbf{Proof of Lemma \ref{lem08}.}
By Lemma \ref{lem02}. for the number of red primes we have
\begin{align*}
    R &\le
    \sum_{\substack{h=2\\ 2\mid h}}^{[r]} 8
    \prod_{2<p}\zj{1-\dfrac{1}{(p-1)^2}}
    \prod_{2<p\mid h^2}\dfrac{p-1}{p-2}
    \cdot \dfrac{x}{(\log x)^2}
    \zj{1+O\zj{\dfrac{\log\log x}{\log x}}}\\
      &+\sum_{\substack{h=1 \\ h\textup{ is odd}}}^{[r]} 1.
  %25. oldal
    \intertext{But $p\mid h^2$ holds if and only if $p\mid h$, thus}
    R&\le
    \sum_{\substack{h=2\\ 2\mid h}}^{[r]} 8
    \prod_{2<p}\zj{1-\dfrac{1}{(p-1)^2}}
    \prod_{2<p\mid h}\dfrac{p-1}{p-2}
    \cdot \dfrac{x}{(\log x)^2}
    \zj{1+O\zj{\dfrac{\log\log x}{\log x}}}\\
    &+\sum_{\substack{h=1 \\ h\textup{ is odd}}}^{[r]} 1.
\end{align*}
From here, we can perform the same calculations as in the proof of
Lemma \ref{lem02}, and we get
\[
R\le 6.646 \dfrac{rx}{(\log x)^2},
\]
which proves \eqref{lem08a}. In order to prove \eqref{lem08b},
consider the primes in $[1,x]$ in the following sequences
\begin{alignat*}{4}
  p_1,&&\ p_{t+1},&&\ p_{2t+1},&&\ \dots\\
  p_2,&& \ p_{t+2},&& \ p_{2t+2},&&\ \dots\\
  &&\vdots && && \\
  p_t,&&\ p_{2t},&&\ p_{3t},&&\ \dots\\
\end{alignat*}
Here, in the $j$-th sequence
\[
2 \le p_j, p_{t+j}, p_{2t+j}, \dots, p_{\ell t+j} \le x,
\]
where $\ell = \ell(j)$ is the largest integer such that $p_{\ell t+j} \le x$, there can be at most $\frac{x}{r^2}$ yellow primes. This follows from the fact that
\[
x > (p_{t+j}-p_j) + (p_{2t+j}-p_{t+j}) + (p_{3t+j}-p_{2t+j}) + \dots + (p_{\ell t+j}-p_{(\ell-1)t+j}),
\]
and in the above sum, at most $\frac{x}{r^2}$ differences can be greater than $r^2$ (otherwise their sum would exceed $x$). Since we have $t$ such sequences, the total number of yellow primes $Y$ is at most $\frac{tx}{r^2}$. This completes the proof of the lemma.
\qed

%26. oldal
\section{Appendix}

Here are the SAGE programs required in the proofs of
Theorems \ref{thm01}, \ref{thm03}, and \ref{thm06}.
I ran these codes on the following online platform: 
\begin{center}
\url{https://sagecell.sagemath.org/}
\end{center}

\bigskip

\begin{samepage}
\noindent ***\\  
\noindent\textbf{Program 1}
\begin{verbatim}
n=5000
prod( 1-1/((p-1)**2) for p in prime_range(3,n)).n()
\end{verbatim}
\textbf{Result:} $0.660175738989977$
\end{samepage}

\begin{samepage}
\bigskip\noindent ***\\  
\noindent\textbf{Program 2}
\begin{verbatim}
T=20000
sum(2/((p-2)*(p+1)) for p in prime_range(3,T+2)).n()
\end{verbatim}
\textbf{Result:} $0.727089417741948$
\end{samepage}

\begin{samepage}
\bigskip\noindent ***\\  
\noindent\textbf{Program 3}
\begin{verbatim}
T=5000
sum(moebius(d)**2/(d**2) for d in range(2, T+1,2)).n()
\end{verbatim}
\textbf{Result:} $0.303923082993008$
\end{samepage}

\begin{samepage}
\bigskip\noindent ***\\  
\noindent\textbf{Program 4}
\begin{verbatim}
T=5000
sum(moebius(d)**2/(2*d**2) for d in range(1, T+2,2)).n()
\end{verbatim}
\textbf{Result:} $0.607886600147474$
\end{samepage}

%27. oldal
\begin{samepage}
\bigskip\noindent ***\\  
\noindent\textbf{Program 5}
\begin{verbatim}
W=prod(p for p in prime_range(2,50))
P=Primes()
E=prod(P.unrank(j) for j in range(0,43))
exponent=(E*log(W,10)).n()
print(log(W,10).n())
print(exponent)
\end{verbatim}
\textbf{Result:} 17.7887972765829, 1.83383491155388e76
\end{samepage}

\end{document}